\documentclass{article}
\usepackage{geometry}
\usepackage{amsmath}
\usepackage{amssymb}
\usepackage{amsthm}

\usepackage[title]{appendix}
\usepackage[numbers,sort&compress]{natbib}
\usepackage[dvipsnames]{xcolor}

\usepackage{subfigure}
\usepackage{setspace}
\usepackage{float}
\usepackage{mathrsfs}
\usepackage{fancyhdr}
\usepackage{graphicx}
\usepackage{xcolor}

\usepackage[colorlinks,            linkcolor=black,            anchorcolor=black,            citecolor=black            ]{hyperref}
\graphicspath{ {./images/} }
\newtheorem{proposition}{Proposition}
\newtheorem{theorem}{Theorem}
\newtheorem{lemma}{Lemma}

\newtheorem{remark}{Remark}
\newtheorem{assumption}{Assumption}

\numberwithin{equation}{section}

\DeclareMathOperator{\diver}{div}
\DeclareMathOperator\Ker{Ker}

\DeclareMathOperator\re{Re}

\DeclareMathOperator\init{init}
\DeclareMathOperator\myin{in}

\newcommand{\p}{\partial}

\title{Exponential convergence to equilibrium for a two-speed model with variant drift fields via the resolvent estimate }

\author{Xu'an Dou\thanks{Beijing International Center for Mathematical Research, Peking University, Beijing, 100871, China (dxa@pku.edu.cn)}\,\, and \,\, Zhennan Zhou\thanks{Beijing International Center for Mathematical Research, Peking University, Beijing, 100871, China (zhennan@bicmr.pku.edu.cn).}}
\begin{document}
	\maketitle

\begin{abstract}
   We study a two-speed model with variant drift fields, which generalizes the Goldstein-Taylor model  but the two advection fields are not necessarily in proportion. Due to the lack of a local equilibrium structure of the steady state, prevailing hypocoercivity methods  could not be directly applied to such a system. To prove the exponential convergence to equilibrium for this model, we use a recent generalization of the Gearhart-Pr\"uss theorem [Sci. China Math. 64 (2021), no. 3, 507–518] to gain semigroup bounds, where the relative entropy estimate plays a central role in gaining desired resolvent estimates via a compactness argument. 
\end{abstract}



\section{Introduction}

We consider the following 
two-speed model with variant advection fields
\begin{equation}\label{eq:system}
    \begin{cases}
    \p_t p_1+\p_x(b_1(x)p_1)=-p_1+p_2,&  x\in(0,1),\,t>0,\\
        \p_t p_2+\p_x(b_2(x)p_2)=p_1-p_2, &x\in(0,1),\,t>0,\\
        b_i(0)p_i(t,0)=b_i(1)p_i(t,1), & i=1,2,\,t>0,\\
        p_i(0,x)=p_{i,\init}(x),& x\in[0,1],\,i=1,2.
    \end{cases}
\end{equation}
Here $p_i(t,x)$ is regarded as the population density of particles of type $i$ at position $x$ and time $t$, in a kinetic description. Each particle is characterized by its spatial coordinate $x\in[0,1]$, and its type $i\in\{1,2\}$. And particles of different types are driven by different velocity fields, namely $b_1(x)$ and $b_2(x)$. The reaction terms on the right hand side indicate that two types of particles can change to each other with a constant transition rate. The boundary condition means that the mass fluxes going out of the domain $[0,1]$ from one side will go back to the domain from the other side, therefore the total mass $$\sum_{i=1,2}\int_0^{1}p_i(t,x)dx$$ is conserved in time. 

We assume that two velocity fields $b_1(x)$ and $b_2(x)$ are both non-degenerate and different from each other in the following sense
\begin{assumption}\label{as:b}
We assume $b_1,b_2\in C^1[0,1]$ 
satisfying 
\begin{equation}
    b_i(x)\neq 0,\quad \forall x\in[0,1],\,i=1,2,
\end{equation} and there exists some $x^*\in[0,1]$ such that
\begin{equation}
    b_1(x^*)\neq b_2(x^*).
\end{equation}
\end{assumption}

We aim to characterize the long time behavior of system \eqref{eq:system}, which is the exponential convergence to the unique steady state. The case $b_1\equiv1,b_2\equiv-1$, widely referred to as the Goldstein-Taylor model, has been studied from various aspects (e.g. \cite{arnold2021large,dolbeault2015hypocoercivity,arnold2020sharpening} and references therein). In this work, however, we are interested in the general scenario including (but not limited to) the cases that $b_1(x)/b_2(x)$ is not a constant.


\paragraph{Motivating Examples} Our primary motivation is a Fokker-Planck equation from neuroscience proposed in \cite{cai2004effective,cai2006kinetic}, which describes the collective behavior of a large number of biological neurons. Each neuron within the target system is characterized by two variables: its voltage (membrane potential) $x$ and conductance $g$. 
In the simplest form,
the Fokker-Planck equation is linear, and is given by
\begin{equation}\label{eq:neuron}
    \p_tp+\p_x(b(x,g)p)=\p_g((g-g_{\myin})p)+\p_{gg}p,\quad x\in(0,1),\,g>0,\,t>0,
\end{equation} where $g_{\myin}>0$ is a constant and the velocity field  $b(x,g)$ is given by
\begin{equation}\label{vel-neuron}
    b(x,g):=-g_Lx+g(V_E-x),
\end{equation} which is a combination of two effects. The first term $-g_Lx$, with a constant $g_L>0$, drives the voltage to its resting potential (which is set to be zero here). Physically it models the leaky effect. The second term $g(V_E-x)$, with $V_E>1$, stimulates the voltage of a neuron to increase, so that it may reach the firing threshold $x=1$. 
If the voltage of a neuron arrives at the firing threshold $x=1$, its value is reset to $0$ immediately, which accounts for the firing-resetting mechanism of  neurons.

Notice that, a firing event could happen only when $b(1,g)>0$, which means the conductance $g$ has to exceed a certain value, i.e., $g>(V_E-1)/g_L$. In this case we assign the following boundary condition
\begin{equation}
    b(0,g)p(t,0,g)=b(1,g)p(t,1,g),\quad g>(V_E-1)/g_L,\, t>0.
\end{equation}
On the other hand, if $g\leq (V_E-1)/g_L$, then $b(1,g)\leq 0$, no neurons at this conductance can reach the firing threshold $x=1$, and therefore no neurons would be reset at $x=0$. In this case the following boundary condition is imposed instead
\begin{equation}
    p(t,0,g)=p(t,1,g)=0,\quad 0<g\leq(V_E-1)/g_L,\,t>0.
\end{equation} Note that in both cases the fluxes at $x=0$ and $x=1$ are equal, i.e., 
$$b(x,g)p(t,x,g) \bigr \vert_{x=0}^{x=1}=0,\quad g>0,\, t>0.$$ Finally, the problem \eqref{eq:neuron} is completed by a simple no-flux boundary condition at $g=0$ 
\begin{equation*}
    (g-g_{\myin})p+\p_gp=0,\quad g=0,\, x\in(0,1),\, t>0.
\end{equation*}

System \eqref{eq:neuron} is a kinetic description of neuron activities and the diffusion is only in the $g$ direction, which is analogous to the velocity variable in typical kinetic equations, and the voltage variable $x$ is reminiscent of the position variable.

First rigorous analysis on this kinetic neuron system \eqref{eq:neuron} is given in \cite{perthame2013voltage}, simplified models are subsequently analyzed in \cite{perthame2018derivation,kim2021fast}, and extensive numerical exploration has been carried out in \cite{caceres2011numerical}. However, due to the difficulties which we shall discuss later in this introduction, no result on the exponential convergence to the steady state has been obtained yet. Clearly, the two-speed model \eqref{eq:system} can be viewed as a simplification of the Fokker-Planck equation \eqref{eq:neuron} where the diffusion in the conductance variable is replaced by a reversible process between two types particles convecting in different velocity fields.

A second motivation is one preliminary attempt to extend the hypoelliptic advection-diffusion equation in \cite{bedrossian2017enhanced} from shear flows to more general flows. We consider the following advection-diffusion equation for $p(t,x,y)$ with $(x,y)\in[0,1]^2$,
\begin{equation}\label{eq:fluid}
    \begin{cases}
    \p_tp+\p_x(b(x,y)p)=\p_{yy}p,\quad &(x,y)\in(0,1)^2,\,t>0,\\
    b(1,y)p(t,1,y)=b(0,y)p(t,0,y),\quad &y\in(0,1),\,t>0,\\
    \p_yp(t,x,0)=\p_yp(t,x,1)=0,\quad &x\in(0,1),\, t>0.
    \end{cases}
\end{equation}
In the shear flow case $b(x,y)\equiv b(y)$, such a system is proposed and analyzed in \cite{bedrossian2017enhanced} to study the enhanced dissipation effects. This shear flow case has also been analyzed in \cite{albritton2021enhanced} recently using H\"{o}rmander's hypoellipicity theory, whose method applies to more general diffusion term like $\p_y(a(y)\p_y\cdot)$ or $\diver(A(y)\nabla\cdot)$ instead of $\p_{yy}$. For an incompressible flow, such as the shear flow, in recent years much understanding has been obtained on its role in such models. But on compressible flows, the related literature seems to be relatively sparse \cite{coti2020homogenization}.



\paragraph{Difficulties}
The above three systems \eqref{eq:system},\eqref{eq:neuron} and \eqref{eq:fluid} can be written into a general form of kinetic equations
\begin{equation}\label{eq:general}
    \p_tp+\mathsf{T}p=\mathsf{C}p,\quad t>0,
\end{equation} where $\mathsf{T}$ denotes a transport operator, $\mathsf{C}$ denotes a collision operator, and certain proper boundary conditions are prescribed. 

In typical kinetic equations such as the Boltzmann equation, the kernel of the collision operator $\mathsf{C}$ is called local equilibriums, and the global equilibrium is in the intersection of the kernels of $\mathsf{T}$ and $\mathsf{C}$. A common difficulty in analyzing the long time behavior of \eqref{eq:general} is its hypocoercivity structure \cite{villani2009hypocoercivity}. Typically the collision operator $\mathsf{C}$ only gives dissipation in partial directions (like the $i$ direction in \eqref{eq:system}, the $g$ direction in \eqref{eq:neuron} and the $y$ direction in \eqref{eq:fluid}), while the transport operator $\mathsf{T}$ itself does not give dissipation. Therefore the operator $\mathsf{T}-\mathsf{C}$ is not coercive with respect to the common $L^2$ norm. And the exponential convergence to the global equilibrium can (only) be obtained via exploring the interplay between $\mathsf{T}$ and $\mathsf{C}$. Such convergence can be obtained via compactness arguments in a non-constructive way in some cases (e.g. \cite{ukai1974existence,bernard2013exponential}). In the last two decades, various hypocoervity methods have been developed to obtain a constructive convergence rate, therefore giving a quantitative characterization of the long time behavior of kinetic equations \cite{arnold2014sharp,dolbeault2015hypocoercivity,villani2009hypocoercivity,armstrong2019variational,baudoin2017bakry}. 

In contrast to other kinetic models, a prominent feature shared by the systems \eqref{eq:system}, \eqref{eq:neuron} or \eqref{eq:fluid} is that the  velocity field is not in a ``separable form''. For example, in \eqref{eq:neuron} the velocity field $b(x,g)=-g_Lx+g(V_E-x)$ can not be written as $b(x,g)=\theta(x)\zeta(g)$, while in most familiar kinetic models one  simply  has $b(x,g)=g$, which is naturally in a separable form. For the two-speed model  \eqref{eq:system}, the velocity field is in a separable form only when $b_i(x)=\theta(x)\zeta(i)$, which is equivalent to $b_1(x)/b_2(x)$ is simply a constant, and such examples include the classical Goldstein-Taylor model when $b_1\equiv1,b_2\equiv-1$. However, in the paper, we aim to study the general two-speed model, where the two velocity fields only need to satisfy the non-degenerate and distinguishable conditions given in Assumption \ref{as:b}. 

The non-separable velocity field  brings new difficulties for understanding the long time asymptotic behavior of the two-speed model in various and yet related aspects.
First and most directly, one can not represent the solution by the method of separation of variables via the Fourier transform in $x$, or more generally, using the eigenfunctions of the operator $\p_x \bigl (\theta(x)\cdot\bigr )$.

Secondly, due to the non-separable velocity fields, the steady states of the systems \eqref{eq:system},\eqref{eq:neuron} or \eqref{eq:fluid} are much more complicated. In particular, the  local equilibrium structure does not exist: the global steady state is no longer in the kernel of both the collision operator $\mathsf{C}$ and the transport operator $\mathsf{T}$. In fact, one can check that in the two-speed model \eqref{eq:system}, the intersection of the kernels of $\mathsf{T}$ and $\mathsf{C}$ contains only trivial solutions (zero) unless $b_1(x)/b_2(x)$ is a constant. Therefore, the global steady state is essentially determined by the integrated combination of the transport operator $\mathsf{T}$, the collision operator $\mathsf{C}$ and (last but not least) the boundary conditions. Thus, due to the lack of a local equilibrium structure and additional challenges that the boundary conditions bring in, it is not clear yet how the prevailing hypocoercivity methods could  be applied to this two-speed model \eqref{eq:system}, although it appears in such a simple form.  

It is worth emphasizing that the challenges in analyzing the two-speed model \eqref{eq:system} already capture the essential difficulties of the kinetic model \eqref{eq:general} with a non-separable velocity field, although more complicated models may face additional difficulties such as  the lack of regularity properties of the steady states. However, the studies on other kinetic models, such as \eqref{eq:neuron} and \eqref{eq:fluid}, are beyond the scope of the current paper. See also the discussions in Section 4.

\paragraph{Main results and methods} In this work, we prove the exponential convergence to the unique global steady state for the two-speed model \eqref{eq:system}. The main results are summarized in the following theorem. 

\begin{theorem}[Main result]\label{thm:main}
Consider the two-speed model  \eqref{eq:system}, where $b_1,\,b_2$ satisfy Assumption \ref{as:b}.  
\begin{enumerate}
    \item There exists a positive steady state $p_{\infty}=(p_{1,\infty},\,p_{2,\infty})$ of \eqref{eq:system} which is unique subject to the normalization condition $\int_0^1(p_{1,\infty}(x)+p_{2,\infty}(x))dx=1$.
    \item For initial data $p_{\init}=(p_{1,\init},\,p_{2,\init})\in (L^2[0,1])^2$, there exists a unique (mild) solution of \eqref{eq:system} which converges to the steady state exponentially. More precisely, there exists $C=e^{\pi/2}>1,\,\alpha>0$ such that 
\begin{equation}\label{main-result}
\left\|p(t)-\Pi p(t)\right\|\leq Ce^{-\alpha t}\left\|p(0)-\Pi p(0)\right\|,\quad \forall t\geq0,
\end{equation}
where $$\Pi p(t):=\left(\int_0^1(p_1(t,w)+p_2(t,w))dw\right)p_{\infty}=\left(\int_0^1(p_{1,\init}(w)+p_{2,\init}(w))dw\right)p_{\infty}.$$
The norm $\|\cdot\|$ in \eqref{main-result} is defined in \eqref{def-inner} using the steady state $p_{\infty}$, and is equivalent to the usual $L^2$ norm.
\end{enumerate}
\end{theorem}
\begin{remark}
Regularity condition $b_1,b_2\in C^1[0,1]$ in Assumption \ref{as:b} can be relaxed to $b_1,b_2\in C[0,1]$. The mild solutions are defined via the $C^0$ semigroup generated by the operator $A$ given in \eqref{def-A}.
\end{remark}


For heuristic purposes, we layout the proof strategy as follows. First, with characterization of the steady states $p_{\infty}=(p_{1,\infty},p_{2,\infty})$, we can carry out the standard relative entropy estimate (\cite{MICHEL20051235,perthame2006transport}), given by 
\begin{equation}\label{real-relative-entropy}
    \frac{d}{dt}\left[\sum_{i=1,2}\int_0^1(h_i(t,x)-1)^2p_{i,\infty}dx\right]=-\int_0^{1}(p_{1,\infty}+p_{2,\infty})|h_1-h_2|^2dx\leq 0,
\end{equation} where $ h_i(t,x):=p_i(t,x)/p_{i,\infty}(x), i=1,2$. Such relative entropy estimate \eqref{real-relative-entropy} gives crucial information on the dynamics of \eqref{eq:system}. In fact, from \eqref{eq:system} one can derive convergence to the steady state in a weak sense by adapting a classical compactness argument (see \cite{MICHEL20051235}, and also Chapter 3.6 of \cite{perthame2006transport}). However, an exponential convergence does not follow from \eqref{real-relative-entropy} due to the lack of coercivity in the dissipation term: it only involves information in the $i$ direction but the dissipation in the $x$ direction is not involved.


Second, to prove the exponential convergence in spite of the degeneracy in the dissipation, we leverage the generalization of the Gearhart-Pr\"uss theorem \cite{helffer2010resolvent,helffer2013spectral} recently shown by Wei \cite{wei2021diffusion}, 
\begin{theorem}[Wei \cite{wei2021diffusion}]\label{thm:wei}
Let $A$ be an m-dissipative operator on a Hilbert space $X$, then the corresponding semigroup $e^{tA}$ satisfies the following bound, 
\begin{equation}\label{semigroup}
    \|e^{tA}\|\leq e^{-t\Psi(A)+\pi/2},
\end{equation} where $\Psi(A)$ is defined as
\begin{equation}\label{resolvent}
    \Psi(A):=\inf \{\|(A-i\lambda)p\|:p\in D(A),\|p\|=1,\lambda\in \mathbb{R}\}.
\end{equation}
\end{theorem} 
Theorem \ref{thm:wei} gives the semigroup bound \eqref{semigroup} via the resolvent estimate \eqref{resolvent}, and has been used in the analysis of hydrodynamic stability (e.g. \cite{chen2020transition}).
We defer giving the definition of an m-dissipative operator \cite{pazy2012semigroups} to Section \ref{sec:2}. We shall show that, the relative entropy estimate \eqref{real-relative-entropy} naturally implies a semigroup formulation of \eqref{eq:system}, where the generator is indeed m-dissipative. Hence, proving the exponential convergence boils down to showing that the corresponding $\Psi(A)>0$, which can be achieved by a compactness argument.  




To sum up, we prove the exponential convergence of the two-speed model with variant drift fields \eqref{eq:system}, by combining the relative entropy structure with a generalized Gearhart-Pr\"uss theorem, to obtain exponentially contracting property on the corresponding time-evolution semigroup. Such strategy has potential to be used in understanding the long time behavior of other (kinetic) models, in particular those with ``non-separable'' advection fields, such as \eqref{eq:neuron} and \eqref{eq:fluid}.

 Throughout this paper, we take Assumption \ref{as:b} unless otherwise specified. The rest of this paper is arranged as follows: With preparations in Section 2: steady state, and the semigroup formulation, we prove Theorem 1 in Section 3. Finally a discussion is given in Section 4. 




\section{Preliminary: steady state and generator of the semigroup} \label{sec:2}

In this section, we aim to construct a proper semigroup formulation for \eqref{eq:system} such that the time evolution semigroup can be shown to be exponentially contacting. However, to this end, we need to first show some preliminary properties of the steady state.


First, we prove the existence and uniqueness of a positive steady state.
\begin{proposition}\label{prop:steady}
There exists a steady state of \eqref{eq:system}, which is unique up to a constant multiplier. Moreover, we can choose a normalization constant such that the steady state $p_{\infty}=(p_{1,\infty},p_{2,\infty})$ is positive and satisfies \begin{equation}\label{steady-intergal}
    \int_0^{1}p_{1,\infty}(x)dx=\int_0^{1}p_{2,\infty}(x)dx=\frac{1}{2}.
\end{equation}
Furthermore, the following $L^{\infty}$ estimate holds
\begin{equation}\label{Linfinity}
    0<c\leq p_{i,\infty}(x)\leq C<+\infty,\quad i=1,2,\ x\in[0,1].
\end{equation}
\end{proposition}
\begin{proof}

The steady state equation is given by \begin{equation}\label{steady-two-speed}
\begin{cases}
\p_x(b_1p_1)=-p_1+p_2,\quad x\in(0,1),\\
\p_x(b_2p_2)=p_1-p_2,\quad x\in(0,1),\\
b_i(0)p_i(0)=b_i(1)p_i(1),\quad i=1,2.
\end{cases}
\end{equation}
It is convenient to introduce the fluxes $J_i=b_ip_i$, $i=1,2$ and thus  
$$\begin{pmatrix}
p_1\\p_2
\end{pmatrix}=\begin{pmatrix}
b_1^{-1}& \\ &b_2^{-1}
\end{pmatrix}\begin{pmatrix}
J_1\\J_2
\end{pmatrix}.$$ 
We can equivalently rewrite the system \eqref{steady-two-speed} as
\begin{equation}\label{steady-ode}
\begin{aligned}
\frac{d}{dx}\begin{pmatrix}
J_1\\J_2
\end{pmatrix}=\begin{pmatrix}
-1& 1\\1&-1
\end{pmatrix}&\begin{pmatrix}
b_1^{-1}& \\ &b_2^{-1}
\end{pmatrix}\begin{pmatrix}
J_1\\J_2
\end{pmatrix},\quad x\in(0,1),\\J:=(J_1,J_2)^{T},&\quad J(0)=J(1).
\end{aligned}
\end{equation}
And the properties of the steady states can be obtained by analyzing the ODE system \eqref{steady-ode}.

Let $\Phi(v)$ be the fundamental matrix of the linear system, i.e.,
\begin{equation}\label{fundamental-v}
\Phi(0)=I,\quad\frac{d}{dx}\Phi(x)=\begin{pmatrix}
-1& 1\\1&-1
\end{pmatrix}\begin{pmatrix}
b_1^{-1}& \\ &b_2^{-1}\end{pmatrix}\Phi(x).\end{equation}
Then to find a steady state is equivalent to find a right-eigenvector of $\Phi(1)$ with eigenvalue $1$. 
The existence follows by observing that $(1,1)$ is a left-eigenvector of $\Phi(x)$ for all $x\in[0,1]$  with eigenvalue $1$, which can be checked directly by computing $\frac{d}{dx}\left((1,1)\Phi(x)\right)$.

For positivity, we note that $J_1(x)+J_2(x)$ is a constant for all $x\in[0,1]$.  Therefore the ODE system \eqref{steady-ode} can be decoupled into
\begin{equation}\label{steady-tmp-pf}
    \frac{d}{dx}J_1(x)=-(b_1^{-1}+b_2^{-1})J_1+b_2^{-1}(J_1(0)+J_2(0)),
\end{equation}with a similar equation for $J_2$. Note that $$\frac{d}{dx}(e^{\int_0^{x}b_1^{-1}+b_2^{-1}dx'}J_1)=e^{\int_0^{x}b_1^{-1}+b_2^{-1}dx'}b_2^{-1}(J_1(0)+J_2(0))$$ and $b_2^{-1}(J_1(0)+J_2(0))$ does not change sign for $x\in[0,1]$, we get that $J_1(x)$ does not change sign from $J_1(0)=J_1(1)$. Similarly $J_2(x)$ does not change sign. Then we deduce that the corresponding $p_1,\,p_2$ does not change sign either. Finally, by integrating the first equation of \eqref{steady-ode} on $[0,1]$, we get $$-\int_0^{1}p_1dx+\int_0^{1}p_2dx=\int_0^{1}\frac{d}{dx}J dx=J_1(1)-J_1(0)=0,$$ and we conclude that $p_1,\,p_2$ are of the same sign for $x\in[0,1]$.

Therefore we can get a non-negative steady state $p_1,p_2$, by multiplying with a proper constant. To show $p_1,p_2$ are indeed positive, we follow a contradiction argument to show $J_1(x),J_2(x)\neq 0$. WLOG suppose $J_1(\bar{x})=0$ for some $\bar{x}\in[0,1]$. If $\bar{x}\in (0,1)$, then $\frac{d}{dx}J_1(\bar{x})$ is also zero since $J_1$ does not change sign. Hence, by the ODE \eqref{steady-ode} we deduce $J_2(\bar{x})=0=J_1(\bar{x})$. Since $(J_1,J_2)$ satisfies the linear ODE \eqref{steady-ode}, it follows that $J_1,J_2\equiv0$, which is a contradiction. If $\bar{x}=0$ or $1$, thanks to the periodic boundary condition for $J$ \eqref{steady-ode} we can derive a contradiction similarly.

Integrating the first equation in \eqref{steady-two-speed} we get
\begin{equation*}
        \int_0^{1}p_{1}(x)dx=\int_0^{1}p_{2}(x)dx>0,
\end{equation*} thanks to the boundary condition. By multiplying a normalization constant we can get a positive steady state $(p_{1,\infty},p_{2,\infty})$ satisfying \eqref{steady-intergal}.

For uniqueness, it is equivalent to show that the kernel of $\Phi(1)-I$ is exactly one-dimensional, where $I$ denotes the identity matrix. Actually, otherwise we would have $\Phi(1)=I$, which means every solution of the ODE in \eqref{steady-ode} will satisfy the condition $J(0)=J(1)$ automatically. By considering initial data $J_1(0)=0,J_2(0)=1$, we get $J_1(1)\neq 0=J_1(0)$ from \eqref{steady-tmp-pf} which is a contradiction.

\end{proof}

 With the steady state $p_{\infty}=(p_{1,\infty},p_{2,\infty})$ given in Proposition \ref{prop:steady}, we define the following complex Hilbert space,
\begin{equation}
    X:=L^2([0,1],p_{1,\infty}^{-1}dx)\times L^2([0,1],p_{2,\infty}^{-1}dx),
\end{equation}whose inner product and norm are given by 
\begin{equation}\label{def-inner}
    (p,q):=\sum_{i=1,2}\int_{0}^{1}p_i\bar{q}_ip_{i,\infty}^{-1}dx,\quad \|p\|:=\left(\sum_{i=1,2}\int_{0}^{1}|p_i|^2p_{i,\infty}^{-1}dx\right)^{1/2}.
\end{equation}Since $p_{i,\infty}^{-1}$ is continuous and positive by Proposition \ref{prop:steady}, the norm in $X$ is equivalent to the usual $L^2$ norm.

We define the unbounded operator $A:D(A)\rightarrow X$
\begin{equation}\label{def-A}
\begin{aligned}
    (Ap)_i&:=-\p_x(b_ip_i)-p_i+p_{i+1},\quad i=1,2,\\
    D(A)&:=\{p=(p_1,p_2)\in X,\, (b_ip_i)\in H^1(0,1), \,b_i(0)p_i(0)=b_i(1)p_i(1),\, i=1,2\},
\end{aligned}
\end{equation} where we use the mod-2 convention for the subscripts: $p_3=p_1$. We shall interpret the solution of \eqref{eq:system} as defined by the semigroup generated by $A$, i.e., mild solution \cite{pazy2012semigroups}.

Next, we translate the relative entropy estimate \eqref{real-relative-entropy} to this semigroup formulation, which shows that $A$ is an m-dissipative operator, which implies that $A$ generates a $C^0$ semigroup of contractions. We recall an operator $\mathcal L$ on a Hilbert space $X$ is called a dissipative operator if $\re(\mathcal Lp,p)\leq 0,$ for all $p\in D(\mathcal L)$. An m-dissipative operator $\mathcal L$ is a dissipative operator with the range $R(\lambda \mathcal I-\mathcal L)=X$ for all $\lambda>0$ \cite{pazy2012semigroups}.

\begin{proposition}\label{prop:dis}
The operator $A$ defined in \eqref{def-A}  is an m-dissipative operator in $X$, with
\begin{equation}\label{relative-entropy}
    \re(Ap,p)=-\frac{1}{2}\int_0^{1}(p_{1,\infty}+p_{2,\infty})|h_1-h_2|^2dx\leq 0,\quad h_i:=p_i/p_{i,\infty},\quad i=1,2.
\end{equation}
\end{proposition}We introduce the following notation:
\begin{equation}
    h_i:=p_{i}/p_{i,\infty},\quad i=1,2,
\end{equation}
Then 
\begin{equation}\label{ap-h}
    (Ap)_i=-\p_x(b_ip_{i,\infty}h_i)-h_ip_{i,\infty}+h_{i+1}p_{i+1,\infty},\quad i=1,2.
\end{equation}
And from the boundary condition of $p$ and $p_{\infty}$ we have
\begin{equation}\label{bc-h}
    h_i(0)=h_i(1),\quad i=1,2.
\end{equation}
\begin{proof}[Proof of Proposition \ref{prop:dis}]
We first check \eqref{relative-entropy} which shows that $A$ is dissipative. Substitute the steady state system \eqref{steady-two-speed} in \eqref{ap-h}, we get
\begin{equation}\label{ap-h2}
    (Ap)_i=-b_ip_{i,\infty}\p_xh_i-h_ip_{i+1,\infty}+h_{i+1}p_{i+1,\infty},\quad i=1,2.
\end{equation}
Therefore
\begin{equation*}
    (Ap,p)=\sum_{i=1,2}\int_0^{1}(-b_ip_{i,\infty}\p_xh_i\bar{h}_i-|h_i|^2p_{i+1,\infty}+h_{i+1}\bar{h}_ip_{i+1,\infty})dx.
\end{equation*}
Taking real part of the first term and integrating by parts we get
\begin{align*}
\re  \int_0^{1}(-b_ip_{i,\infty}\p_xh_i\bar{h}_i)dx&=\int_0^{1}(-b_ip_{i,\infty}\p_x(\frac{1}{2}|h_i|^2))dx=\int_0^{1}\frac{1}{2}|h_i|^2\p_x(b_ip_{i,\infty})dx\\&=\int_0^{1}\frac{1}{2}|h_i|^2(p_{i+1,\infty}-p_{i,\infty}).
\end{align*}
Therefore we continue the calculation to conclude that $A$ is dissipative 
\begin{align*}
    \re (Ap,p)&=\re \sum_{i=1,2}\int_0^{1}(\frac{1}{2}|h_i|^2(p_{i+1,\infty}-p_{i,\infty})-|h_i|^2p_{i+1,\infty}+h_{i+1}\bar{h}_ip_{i+1,\infty})dx\\&=-\frac{1}{2}\int_0^{1}(p_{1,\infty}+p_{2,\infty})|h_1-h_2|^2dx\leq 0.
\end{align*}
To show that $A$ is m-dissipative, it suffices to show that its adjoint $A^*$ is also dissipative (\cite{pazy2012semigroups} Chapter 1.4 Corollary 4.4). With the boundary condition of $p_{\infty}$, it is standard and straightforward to check that $D(A^*)=D(A)$, therefore
\begin{equation*}
    \re(A^*p,p)=\re(p,Ap)=\re(Ap,p)\leq 0,\quad \forall p\in D(A^*).
\end{equation*}
\end{proof}

Note that we can not find an $\alpha>0$ such that $\re(Ap,p)\leq -\alpha\|p\|^2$, which reflects the lack of coercivity in \eqref{eq:system}.

As a final preparation, we show that there is no eigenvalue other than $0$ has a non-negative real part.
\begin{proposition}\label{prop:eigen}
If $\lambda$ is an eigenvalue of the operator $A$ with a non-negative real part and $p$ is an associated eigenfunction. Then $\lambda=0$ and $p$ is $p_{\infty}$ multiplied by a constant. 
\end{proposition}
\begin{proof}
By the entropy estimate \eqref{relative-entropy} in Proposition \ref{prop:dis}, we deduce
\begin{equation*}
    0\leq \re (\lambda p,p)=\re (A p,p)=-\frac{1}{2}\int_0^{1}(p_{1,\infty}+p_{2,\infty})|h_1-h_2|^2dx\leq 0.
\end{equation*}
Therefore $h_1=h_2=:\bar{h}$, from \eqref{ap-h2} we obtain
\begin{equation*}
        (Ap)_i=-b_ip_{i,\infty}\p_x\bar{h},\quad i=1,2.
\end{equation*} Since $p$ is an eigenfunction, $(Ap)_i=\lambda p_i=\lambda \bar{h}p_{i,\infty}$, we get
\begin{equation}\label{eigen-tmp-pf}
-b_i\p_x\bar{h}=\lambda \bar{h},\quad i=1,2.    
\end{equation}
From Assumption \ref{as:b} $b_1\neq b_2$ at some $x^*\in[0,1]$, therefore $\bar{h}(x^*)=0$ if $\lambda\neq 0$. In this case, by a basic property of linear ODE, we conclude that $\bar{h}\equiv0$. 

If $\lambda=0$, from \eqref{eigen-tmp-pf} we deduce $\p_x\bar{h}=0$ which implies that $\bar{h}$ is a constant, i.e., $p$ is $p_{\infty}$ multiplied by a constant. 
\end{proof}
Actually, the last argument in the proof above serves as an alternative proof of the uniqueness of the steady state which is based on the entropy estimate, rather than the ODE argument in Proposition \ref{prop:steady}.

\section{Proof of the main theorem}

To obtain the exponentially contracting property of the semigroup $e^{tA}$, in the view of Theorem \ref{thm:wei}, we shall work in the orthogonal complement of $\Ker{A}=\text{span}\{p_{\infty}\}$ in $X$, denoted as $X_0$:
\begin{equation}\label{def-X0}
    X_0:=\{p\in X,\ (p,p_{\infty})=0\}.
\end{equation}
By definition of the inner product in $X$ \eqref{def-inner}, we have $(p,p_{\infty})=\sum_{i=1,2}\int_0^{1}(p_1+p_2)dx$, therefore $p\in X$ is in $X_0$ if and only if 
\begin{equation}
\int_{0}^{1}(p_1+p_2)dx=0.
\end{equation}
Clearly, $X_0$ is a Hilbert space. Note that as an operator on $X$, $A$ actually maps $D(A)$ to $X_0$, which allows us to restrict it as an unbounded operator on $X_0$. The following Lemma translates the result on $A$ in the previous section to its restriction on $X_0$.
\begin{lemma}\label{Lemma:1}
The restriction of $A$ on $X_0$: $D(A)\cap X_0\rightarrow X_0$, is m-dissipative and has no eigenvalue with a non-negative real part.
\end{lemma}
\begin{proof}
By Proposition \ref{prop:eigen} and the definition of $X_0$ \eqref{def-X0}, we deduce the statement on eigenvalue.

Since $X_0$ is a subspace of $X$, it is easy to see $A$ is still dissipative on $X_0$. To show $A$ is m-dissipative on $X_0$, it remains to check that for all $\lambda>0$, it holds that for all $f$ in $X_0$, there exists $u$ in $X_0$ such that $(\lambda-A)u=f$. Since $A$ is m-dissipative in $X$ by Proposition \ref{prop:dis}, there exists some $u\in X$ satisfying this. By a straightforward integration of the equation $(\lambda-A)u=f$, we get $$\lambda\int_{0}^{1}(u_1+u_2)dx=\int_{0}^{1}(f_1+f_2)dx=0.$$
This shows that $u$ is actually in $X_0$, which finishes the proof.
\end{proof}
Now we begin the proof of Theorem 1.
\begin{proof}[Proof of Theorem 1]
Notice that the first part of the theorem has already been proved in Proposition \ref{prop:steady}.

For the second part, by Theorem \ref{thm:wei}, it suffices to show that $\Psi(A)$ defined as follows is strictly positive,
\begin{equation}
    \Psi(A):=\inf \{\|(A-i\lambda)p\|:p\in D(A)\cap X_0,\|p\|=1,\lambda\in \mathbb{R}\}.
\end{equation}

We argue by contradiction. The idea is to use a compactness argument to derive a contradiction with Assumption \ref{as:b}. 

Suppose $\Psi(A)=0$, then there exists a real sequence $\{\lambda_n\}$ and a sequence $\{p_n\}$ such that $p_n\in D(A)\cap X_0$, $\|p_n\|=1$ and we have
\begin{equation}\label{pf-tmp-0}
   Ap_n-i\lambda_n p_n:=f_n\rightarrow0,\quad in\ X_0.
\end{equation}
Recall in Lemma \ref{Lemma:1}, the operator $A$ can be restricted as a m-dissipative operator on $X_0$. And the norm of $X_0$, inherited from $X$ \eqref{def-inner}, is equivalent to the usual $L^2$ norm.

We divide the discussion into two cases.

Case 1: $\{\lambda_n\}$ is bounded. Up to a subsequence, we might as well take $\lambda_n$ convergence to $\lambda^*\in\mathbb{R}$. In this case
\begin{equation}\label{tmp-pf}
    \p_x(b_jp_{j,n})=i\lambda_n p_{j,n}+p_{j+1,n}-p_{j,n}+f_{j,n},\quad j=1,2,
\end{equation} is also bounded in $L^2$. Apply the compact embedding from $H^1$ to $L^2$ on $b_jp_{j,n}$, we can take $b_jp_{j,n}$ converges strongly in $L^2$ up to extraction of subsequences. Therefore by Assumption \ref{as:b}, we get $p_n$ converges to some $p^*\in X_0$ strongly in $L^2$. Then using \eqref{tmp-pf} again we obtain the strong convergence of $\p_x(b_jp_{j,n})$ in $L^2$. Now taking the limit in \eqref{tmp-pf}, we deduce that $p^*\in D(A)\cap X_0$ and
\begin{equation*}
    Ap^*=i\lambda^* p^*.
\end{equation*}  Note that $p^*\neq0$ from the strong convergence of $p_n$. However the above equation shows that $p^*\in X_0$ is an eigenfunction of $A$ with the eigenvalue $i\lambda^*$, which contradicts with Lemma \ref{Lemma:1}. 

Case 2: When $\{\lambda_n\}$ is not bounded, it suffices to consider the case $\lambda_n\rightarrow+\infty$ without loss of generality. By the relative entropy estimate \eqref{relative-entropy},
\begin{align*}
    \frac{1}{2}\int_0^{1}(p_{1,\infty}+p_{2,\infty})|h_{1,n}-h_{2,n}|^2dx&=-\re (Ap,p)\\&=-\re(i\lambda_np_n+f_n,p_n)\\&=-\re(f_n,p_n)\leq \|f_n\|\|p_n\|\rightarrow 0,
\end{align*} as $n$ goes to infinity. Therefore we have
\begin{equation}\label{con-h}
    h_{1,n}-h_{2,n}\rightarrow 0,\quad \text{in}\ L^2.
\end{equation}Then we write out \eqref{pf-tmp-0} in terms of $h$,
\begin{equation*}
    -i\lambda_np_{1,\infty}h_{1,n}+p_{2,\infty}(h_{2,n}-h_{1,n})-(\p_xh_1)b_1p_{1,\infty}=f_{1,n},
\end{equation*} combined which with \eqref{con-h} we deduce
\begin{equation*}
    -i\lambda_nh_1-\p_xh_1b_1\rightarrow0,\quad \text{in }L^2,
\end{equation*} which is equivalent to 
\begin{equation}\label{con-eh}
    \p_x\left(\exp\left(\int_0^{x}i\lambda_nb_1^{-1}dw\right)h_{1,n}\right)\rightarrow 0,\quad \text{in }L^2.
\end{equation}
Using \eqref{con-eh} and the compact embedding from $H^1$ to $L^2$ on $\exp\left(\int_0^{x}i\lambda_nb_1^{-1}dw\right)h_{1,n}(x)$, we deduce that there exists some constant $c_1$ such that
\begin{equation}\label{pf-tmp-h1}
    \exp\left(\int_0^{x}i\lambda_nb_1^{-1}dw\right)h_{1,n}-c_1\rightarrow 0,\quad \text{in }L^2.
\end{equation}
Similarly we deduce that for $h_2$, there exists some constant $c_2$ such that
\begin{equation}\label{pf-tmp-h2}
    \exp\left(\int_0^{x}i\lambda_nb_2^{-1}dw\right)h_{2,n}-c_2\rightarrow 0,\quad \text{in }L^2.
\end{equation}
By the upper bound on $p_{\infty}$ in Proposition \ref{prop:steady}, we get $\|h_n\|\geq \frac{1}{C}\|p_n\|=\frac{1}{C}$ for some $C>0$ and therefore $|c_1|+|c_2|\neq0$ (otherwise $h_n$ strongly converges to zero in $L^2$). 

In light of \eqref{con-h}, \eqref{pf-tmp-h1} and \eqref{pf-tmp-h2}, we derive that
\begin{equation}\label{pf-tmp-3}
    c_1-\exp\left(\int_0^{x}i\lambda_n(b_1^{-1}-b_2^{-1})dw\right)c_2\rightarrow 0,\quad \text{in }L^2.
\end{equation}

Now we calculate the standard $L^2$ norm $\|\cdot\|_{L^2[0,1]}$
\begin{align*}
    \left\| c_1-\exp\left(\int_0^{v}i\lambda_n(b_1^{-1}-b_2^{-1})dw\right)c_2\right\| _{L^2[0,1]}^2&\geq |c_1|^2+|c_2|^2\\&\ \ \ -2|c_1c_2|\left|\int_0^{1}\exp\left(\int_0^{x}i\lambda_n(b_1^{-1}-b_2^{-1})dw\right)dx\right|\\&\geq |c_1|^2+|c_2|^2-2|c_1c_2|\geq 0.
\end{align*}

By the calculation above together with the limit \eqref{pf-tmp-3}, we deduce that $|c_1|=|c_2|\neq 0$ and
\begin{equation}\label{pf-tmp-4}
    \left|\int_0^{1}\exp\left(\int_0^{x}i\lambda_n(b_1^{-1}-b_2^{-1})dw\right)dx\right|\rightarrow 1,\quad 
\end{equation} 
as $n$ goes to infinity.

By Assumption \ref{as:b}, as a continuous function,  $b_1^{-1}-b_2^{-1}\neq 0$ at some $x^*\in[0,1]$. Then we conclude that \eqref{pf-tmp-4} is impossible with $\lambda_n\rightarrow+\infty$ by a stationary phase estimate, see Lemma \ref{Lemma:stationary phase} below.
\end{proof}
\begin{lemma}\label{Lemma:stationary phase}
For a nonzero, real-valued function $\psi\in C[0,1]$ and a real sequence $\lambda_n\rightarrow +\infty$, as $n$ goes to infinity, we have
\begin{equation}\label{pf-tmp-5}
    \limsup_{n\rightarrow \infty}\left|\int_0^{1}\exp\left(\int_0^{x}i\lambda_n\psi(w)dw\right)dx\right|<1.
\end{equation}
\end{lemma}
\begin{proof}[Proof of Lemma \ref{Lemma:stationary phase}]
We can find an interval $[a,b]\subset[0,1]$ where $\psi$ does not vanish.
Since $|e^{\int_0^{x}i\lambda_n\psi dw}|\leq 1$, it suffices to show that 
\begin{equation}\label{tmp-lm2}
     \lim_{n\rightarrow \infty}\left|\int_a^{b}\exp\left(\int_0^{x}i\lambda_n\psi(w)dw\right)dx\right|=0<b-a.
\end{equation}
We assume $\psi(x)>0$ on $[a,b]$ without loss of generality. Consider a change of variable $dy=\psi(x)dx$, $y=\int_a^{x}\psi(w)dw$, then \eqref{tmp-lm2} is equivalent to 
\begin{equation*}
         \lim_{n\rightarrow \infty}\left|\int_0^{\bar{y}}e^{i\lambda_ny}\frac{1}{\psi(x(y)))}dy\right|=0,
\end{equation*} where $\bar{y}=\int_a^{b}\psi(w)dw$ and $x(y)$ stands for the inverse change of variable from $dy=\psi(x)dx$. Then the result follows from the Riemann-Lebesgue lemma since $\frac{1}{\psi(x(y)))}$ is a continuous function on $[0,\bar{y}]$.
\end{proof}
\section{Discussion}

In this article, we prove the exponential convergence to equilibrium of \eqref{eq:system}. We use a non-constructive compactness argument, which does not give an explicit decay rate. As explained in the introduction, it seems difficult to adapt constructive hypocoercivity methods in literature to \eqref{eq:system}. We may study more explicit estimates on the decay rate in the future.

The assumption (Assumption \ref{as:b}) on the difference between $b_1,b_2$ is fairly weak. In fact, this condition is sharp: if otherwise $b_1\equiv b_2=:b$ then $\bar{p}:=p_1+p_2$ would satisfy a pure transport equation
\begin{equation*}
\begin{cases}
    \p_t\bar{p}+\p_x(b(x)\bar{p})=0,\quad &x\in(0,1),\, t>0,\\
    b(0)\bar{p}(t,0)=b(1)\bar{p}(t,1),\quad & t>0.
\end{cases}
\end{equation*} Therefore no convergence to equilibrium can be expected in general. Actually, in this case two kinds of particles are indistinguishable in the sense that they are driven by a same velocity field. In view of the general formulation \eqref{eq:general}, when $b_1\equiv b_2$, the transport operator $\mathsf{T}$ and the collision operator $\mathsf{C}$ commute, therefore no interplay can be used to compensate the lack of coercivity.

\paragraph{Extension to ``continuous'' systems \eqref{eq:neuron},\eqref{eq:fluid}}
Although our approach for the two-speed model \eqref{eq:system} serves as a preliminary study in treating a non-separable velocity field, extending the result to \eqref{eq:neuron} or \eqref{eq:fluid} is far from straightforward. Our proof of Theorem \ref{thm:main} benefits from the fact that the collision operator $\mathsf{C}$ in \eqref{eq:system} is a bounded operator, which is not the case in \eqref{eq:neuron} or \eqref{eq:fluid}. The essential difficulty may be that we have to deal with a \textit{hypoelliptic} boundary value problem, either in the steady state equation $Ap_{\infty}=0$ or in the resolvent estimate $(A-i\lambda)u=f$. However, for a {hypoelliptic} boundary value problem, a general well-posedness and regularity theory is lacking \cite{armstrong2019variational}, in contrast to their elliptic counterparts.

\paragraph{Extension with a non-homogeneous cross-section}
Another interesting extension is to consider a non-homogeneous cross-section $\sigma(x)\geq 0$ on the collision term in \eqref{eq:system}
\begin{equation}\label{extension:system}
    \begin{cases}
    \p_t p_1+\p_x(b_1(x)p_1)=-\sigma(x)(p_1-p_2),&  x\in(0,1),\,t>0,\\
        \p_t p_2+\p_x(b_2(x)p_2)=\sigma(x)(p_1-p_2), &x\in(0,1),\,t>0,\\
        b_i(0)p_i(t,0)=b_i(1)p_i(t,1), & i=1,2,\,t>0,\\
        p_i(0,x)=p_{i,\init}(x),& x\in[0,1],\,i=1,2.
    \end{cases}
\end{equation}
For the Goldstein-Taylor model, i.e., $b_1\equiv1,b_2\equiv-1$, such an extension has been studied in \cite{dietert2021finding,bernard2013optimal,arnold2021large,tran2013convergence}. For non-homogeneous cases, the approach for \eqref{eq:system} may directly extend to \eqref{extension:system} provided the following assumption on $b_1,b_2,\sigma$ holds,
\begin{assumption}\label{as:extend}
We assume $b_1,b_2,\sigma \in C[0,1]$ satisfying 
\begin{equation}
    b_i(x)\neq0,\, i=1,2,\quad \sigma(x)\geq 0,\quad\forall x\in[0,1],
\end{equation} and there exists $x^*\in[0,1]$ such that 
\begin{equation}
    b_1(x^*)-b_2(x^*)\neq 0,\quad \sigma(x^*)\neq 0.
\end{equation}
\end{assumption}
In Assumption \ref{as:extend} we allow degeneracy of $b_1-b_2$ and $\sigma$, but they need to be simultaneously non-zero at one point $x^*$. Such an assumption can not be relaxed to that there exists two (possibly different) points $x^*,y^*\in[0,1]$ such at $b_1-b_2\neq 0$ at $x^*$ and $\sigma>0$ at $y^*$. Otherwise, one can construct examples such that the operator $A$ has a pure imaginary eigenvalue, following the proof of Proposition \ref{prop:eigen}. 

\section*{Acknowledgement}
Z. Zhou is supported by the National Key R\&D Program of China, Project Number 2021YFA1001200, 2020YFA0712000 and NSFC grant Number 12031013, 12171013. X.Dou is partially supported by The Elite Program of
Computational and Applied Mathematics for PhD Candidates in Peking University. We also thank
Jos\'e Carrillo, Beno\^it Perthame,  Xiaoqian Xu and Zhifei Zhang for helpful discussions.

\bibliography{resolvent.bib}
\bibliographystyle{plain}
\end{document}